\documentclass[12pt, a4paper,reqno]{amsart}

\usepackage[euler-digits]{eulervm}
\usepackage{graphicx,enumitem}
\usepackage{amsfonts, amsthm, amssymb, amsmath, stmaryrd}
\usepackage{mathrsfs,array}
\usepackage{eucal,times,color,accents}
\usepackage{url}
\usepackage{float}
\usepackage{pbox}
\usepackage[headings]{fullpage}

\usepackage{ytableau}
\usepackage{color}
\usepackage{mathrsfs}
\usepackage{amssymb}
\usepackage{bm}
\usepackage{amssymb}
\usepackage{ulem}
\usepackage{hyperref}
\usepackage[all,cmtip]{xy}
\usepackage{comment}

\newcommand{\ncom}{\newcommand}

\ncom{\dho}{\partial}
\ncom{\rar}{\rightarrow}
\ncom{\imply}{\Rightarrow}
\ncom{\lrar}{\longrightarrow}
\ncom{\into}{\hookrightarrow}
\ncom{\onto}{\twoheadrightarrow}
\ncom{\ov}{\overline}
\ncom{\m}{\mbox}
\ncom{\sta}{\stackrel}
\ncom{\invlim}{\varprojlim}
\ncom{\xhat}{\widehat}

\ncom{\vspc}{\vspace{3mm}}
\ncom{\End}{{\cE}nd}
\ncom{\tensor}{\otimes}

\ncom{\al}{\alpha}
\ncom{\cHom}{{\mathcal Hom}}

\ncom{\A}{{\mathbb A}}
\ncom{\comx}{{\mathbb C}}
\ncom{\E}{{\mathbb E}}
\ncom{\F}{{\mathbb F}}
\ncom{\G}{{\mathbb G}}
\ncom{\K}{{\mathbb K}}
\ncom{\Le}{{\mathbb L}}
\ncom{\N}{{\mathbb N}}

\ncom{\p}{{\mathbb P}}
\ncom{\Q}{{\mathbb Q}}
\ncom{\R}{{\mathbb R}}
\ncom{\Z}{{\mathbb Z}}

\ncom{\f}{\dfrac}

\ncom{\wtil}{\widetilde}

\ncom{\ci}{{\mathpzc i}}

\ncom{\cA}{{\mathcal A}}
\ncom{\cC}{{\mathcal C}}
\ncom{\cE}{{\mathcal E}}
\ncom{\cF}{{\mathcal F}}
\ncom{\cG}{{\mathcal G}}
\ncom{\cH}{{\mathcal H}}
\ncom{\cI}{{\mathcal I}}
\ncom{\cJ}{{\mathcal J}}
\ncom{\cK}{{\mathcal K}}
\ncom{\cL}{{\mathcal L}}
\ncom{\cM}{{\mathcal M}}
\ncom{\cN}{{\mathcal N}}
\ncom{\cO}{{\mathcal O}}
\ncom{\cP}{{\mathcal P}}
\ncom{\cQ}{{\mathcal Q}}
\ncom{\cR}{{\mathcal R}}
\ncom{\cS}{{\mathcal S}}
\ncom{\cT}{{\mathcal T}}
\ncom{\cU}{{\mathcal U}}
\ncom{\cV}{{\mathcal V}}
\ncom{\cW}{{\mathcal W}}
\ncom{\cX}{{\mathcal X}}
\ncom{\cY}{{\mathcal Y}}
\ncom{\cZ}{{\mathcal Z}}

\ncom{\cSU}{{\mathcal S \mathcal U}}
\ncom{\eop}{{\hfill $\Box$}}
\ncom{\isom}{\cong}

\theoremstyle{plain}
\newtheorem{theorem}{Theorem}
\newtheorem{lemma}[theorem]{Lemma}

\newtheorem{corollary}[theorem]{Corollary}

\theoremstyle{definition}

\newtheorem{defn}{Definition}

\theoremstyle{remark}
\newtheorem{remark}{Remark}

\long\def\comment#1{}

\begin{document}

\title{On Auslander's depth formula}

\author{Shashi Ranjan Sinha}
\address{Department of Mathematics, Indian Institute of Technology -- Hyderabad, 502285, India.}
\email{ma20resch11005@iith.ac.in}

\author{Amit Tripathi}
\address{Department of Mathematics, Indian Institute of Technology -- Hyderabad, 502285, India.}
\email{amittr@gmail.com}
\subjclass[2010]{13D07, 13C14, 13C15}

\begin{abstract} We show that if Auslander's depth formula holds for non-zero Tor-independent modules over Cohen-Macaulay local rings of dimension $1$, then it holds for such modules over any Cohen-Macaulay local ring. More generally, we show that the depth formula for non-zero Tor-independent modules which have finite Cohen-Macaulay dimension over depth $1$ local rings implies the depth formula for such modules over any positive depth local ring.
\end{abstract}

\keywords{Auslander, depth formula, tensor product}
\date{\today}
\maketitle

All rings are assumed to be commutative Noetherian local and all modules are assumed to be finitely generated. We say that a pair $(M,N)$ of $R$-modules satisfies Auslander's depth formula (or simply \textit{the depth formula}) if \begin{align*} 
	depth\, M + depth\, N = depth\, R + depth\, M \otimes_R N.
\end{align*}
We say that $M$ and $N$ are {Tor-independent over $R$} if $Tor_i^R(M,N) = 0$ for all $i \geq 1.$  

Auslander \cite{A1} proved the depth formula under the assumption that the $R$-modules $M$ and $N$ are Tor-independent and the projective dimension of $M$ is finite. This formula was shown to hold for Tor-independent modules over complete intersection local rings by Huneke and Wiegand \cite{HW}.  Their result was generalized to Tor-independent modules over arbitrary local rings, under the additional assumption that one of the modules has finite complete intersection dimension, by Araya and Yoshino \cite{ArYos}, and independently by Iyengar \cite{SI}. Gheibi, Jorgensen, and Takahashi \cite{GJT} have proved the depth formula for Tor-independent modules when one of the modules has a finite quasi-projective dimension.

More generally, we can ask if $(M, N)$ is a pair of $R$-modules satisfying a Tor vanishing condition (Absolute, Relative or Tate Tor, see Avramov and Martsinkovsky\cite{AM}) such that one or both the modules have a finite homological dimension, then the pair $(M, N)$ satisfies the depth formula. In this direction, various versions have been established - see Bergh and Jorgensen \cite{BerJo2}, Celikbas, Liang, and Sadeghi \cite{CLS}, and Christensen and Jorgensen \cite{CJ}, for instance.

In this note, we restrict to the "classical" version of the depth formula where the Tor vanishing condition is for absolute Tor. In this form, it is unknown whether the depth formula is true, even for Gorenstein local rings. The most general result seems to be \cite{CJ} which proves the depth formula over $AB$-rings.

In \cite[Example 2.9]{ArYos} Araya and Yoshino construct $R$-modules $M$ and $N$, where $R$ is a formal power series ring in 5 variables over a field $k$, such that $pd_R(N) = 1$ and $Tor_i(M, N) = 0$ for $i \geq 2,$ but $depth \, M + depth\, N \neq depth\, R + depth\, M\otimes_R N.$ This shows that the Tor-independence condition can not be relaxed.

We show that,

\begin{theorem} \label{corollary_main}
	If the depth formula holds for non-zero Tor-independent modules over Cohen-Macaulay local rings of dimension 1, then it holds for such modules over any Cohen-Macaulay local ring.  
\end{theorem}

This result follows as a corollary to the more general result stated below.
\begin{theorem} \label{theorem_main} 
	If the depth formula holds for non-zero Tor-independent modules with finite Cohen-Macaulay dimension over local rings with depth $1$, then it holds for such modules over any local ring with positive depth.  
\end{theorem}

\begin{remark}
	The proof of Theorem \ref{corollary_main} will work if we assume the underlying ring to be a Gorenstein ring.  The proof of Theorem \ref{theorem_main} will work if replace  "Cohen-Macaulay dimension" with other homological dimensions like  Gorenstein dimension $G\mbox{-}dim_R(M)$, upper Gorenstein dimension $G^*\mbox{-}dim_R(M)$, lower complete intersection dimension $CI_*\mbox{-}dim_R(M)$, or any other homological dimension which satisfies some standard properties (see Theorem \ref{theorem_H_dim}). We have demonstrated the main idea using the Cohen-Macaulay dimension as it is the most general homological dimension (see Theorem \ref{theorem_H_dim_inequalities}) in the list given above.
\end{remark}


\subsection*{Acknowledgements} The authors thank Jishnu Biswas and Suresh Nayak for helpful comments and suggestions. The first author was partially supported by a CSIR senior research fellowship through the grant no. 09/1001(0097)/2021-EMR-I. The second author was partially supported by a Science and Engineering Research Board (SERB) grant MTR/2020/000164.

\section{Preliminaries}
Throughout, $R$ denotes a local ring $(R,m,k)$ and all $R$-modules are finitely generated. For any $R$-module $W$ and an element $x \in m$, we define $\overline{W}_x :=W \otimes_R R/xR.$ When $x$ is clear from the context, we will write simply $\overline{W}.$  

We will assume that $M$ and $N$ are $R$-modules with following minimal free resolutions
\begin{align}
\label{eqn_M} \cdots \xrightarrow{f_3} P_2 \xrightarrow{f_2} P_1 \xrightarrow{f_1}  P_0 \xrightarrow{f_0} M \rar 0 \\ 
\label{eqn_N} \cdots \xrightarrow{g_3} Q_2 \xrightarrow{g_2} Q_1 \xrightarrow{g_1}  Q_0 \xrightarrow{g_0} N \rar 0
\end{align} Let $\Omega_i(M)$ be the $i$'th syzygy of $M.$ In particular, $\Omega_1(M) = Ker(P_0 \rar M).$ Let $\Omega_i(N)$ be the $i$'th syzygy of $N.$

For any $R\mbox{-}$module $M$, there are several possible homological dimensions such as its projective dimension $pd_R(M)$, Gorenstein dimension $G\mbox{-}dim_R(M)$ \cite{AB} , complete intersection dimension  $CI\mbox{-}dim_R(M)$ \cite{AGP}, lower complete intersection dimension  $CI_{*}\mbox{-}dim_R(M)$ \cite{Gerko} , upper Gorenstein dimension  $G^{*}\mbox{-}dim_R(M)$ \cite{Veliche} or Cohen-Macaulay dimension $CM\mbox{-}dim_R(M)$  \cite{Gerko}. These homological dimensions share some common properties which we summarize below.
\begin{theorem}\cite[Theorem 8.7]{Av} \label{theorem_H_dim}
	Let $M \neq 0$ be an $R$-module. Let $H\mbox{-}dim_{R}(M)$ be a homological dimension of $M$ where $H\mbox{-}dim$ can be  $G\mbox{-}dim$, $CM\mbox{-}dim$, $CI\mbox{-}dim$, $CI_*\mbox{-}dim$, or $G^*\mbox{-}dim.$ 
	\begin{enumerate}[label=(\alph*)]

		\item \label{thm_AB}  If $H\mbox{-}dim_{R} (M) < \infty$, then $depth\, M + H\mbox{-}dim_{R} (M) =  depth\, R.$ In particular, $H\mbox{-}dim_{R} (M) \leq depth\, R.$ 
		\item $H\mbox{-}dim_R(\Omega_n(M)) = \max\{H\mbox{-}dim_R(M) - n,0\}.$ 
	 \item If $x \in m$ be an element regular on $R$ and $M$, then
		 \begin{enumerate}[label=(\alph*)]
			 		\item $H\mbox{-}dim_{R}(\overline{M}) = H\mbox{-}dim_R(M) +1.$
		 			\item  If $H\mbox{-}dim_{R} (M) < \infty$, then  $H\mbox{-}dim_{\overline{R}}(\overline{M}) = H\mbox{-}dim_R(M).$  
		 \end{enumerate}

	\end{enumerate}
\end{theorem}
\begin{proof}
The assertion for $H= G$ is from \cite{AB} (and \cite{Masek} for a corrected proof of \ref{thm_AB}). The proof for $H = CI, G^*,$ and $CM/CI_*$ can be found in \cite{AGP}, \cite{Veliche}, and \cite{Gerko}, respectively. 
\end{proof}

\begin{remark}
If $H\mbox{-}dim_{R} (M), H\mbox{-}dim_{R} (N)$ and $H\mbox{-}dim_{R} (M \otimes_R N)$ are all finite, then by Theorem \ref{theorem_H_dim}   \ref{thm_AB}, we can also restate the depth formula as $$H\mbox{-}dim_{R} (M)  + H\mbox{-}dim_{R} (N)  = H\mbox{-}dim_{R} (M \otimes N).$$
\end{remark}

These homological dimensions satisfy the following inequalities  

\begin{theorem}\cite[Theorem 8.8]{Av} \label{theorem_H_dim_inequalities}
	With the notation above, we have
	\begin{enumerate}
		\item $CM\mbox{-}dim_R(M) \leq  G\mbox{-}dim_R(M)  \leq G^*\mbox{-}dim_R(M) \leq CI\mbox{-}dim_R(M).$ 
		\item $CM\mbox{-}dim_R(M) \leq G\mbox{-}dim_R(M) \leq CI_*\mbox{-}dim_R(M) \leq CI\mbox{-}dim_R(M).$ 
	\end{enumerate} where finiteness at any point implies that all the inequalities to its left are equalities.
\end{theorem}
\begin{proof} These inequalities follow from the definition of the respective dimensions, except for the lower bound for $G^*\mbox{-}dim_R(M)$ which is from \cite{Veliche}, and the upper bound for $CI_*\mbox{-}dim_R(M)$ which is from \cite{AGP}. 
\end{proof}

\subsection{Cohen-Macaulay dimension} For the sake of completion, we discuss the definition of Cohen-Macaulay dimension and refer the reader to Gerko \cite{Gerko} for details.

Recall that the $\it{grade}$ of an $R\mbox{-}$module $M$, denoted by $grade(M)$, is defined to be the smallest integer $i$ such that $Ext_{R} ^{i} (M,R) \neq 0$. It is easy to see that $grade(M) \leq G\mbox{-}dim_{R} (M).$ If $grade(M) = G\mbox{-}dim_{R} (M)$, then $M$ is said to be $G\mbox{-} \it{perfect}.$ We say an ideal $I$ is $G\mbox{-} \it{perfect}$ over $R$ if $R/I$ is a $G\mbox{-} \it{perfect}$ $R$-module.

\begin{defn}
	A $G\mbox{-}quasi\mbox{-}deformation$ of a ring $R$ is a diagram of local homomorphisms $R \rightarrow R' \leftarrow Q,$ where $R \rightarrow R'$ is a flat extension and $Q \rightarrow R'$ is a $G\mbox{-}deformation$, that is, a surjective homomorphism whose kernel is a $G\mbox{-}\it{perfect}$ ideal. 
\end{defn}

\begin{defn}
	The Cohen-Macaulay dimension of an $R\mbox{-}$module $M$ is defined as \begin{align*}
			CM\mbox{-}dim_{R} (M) =& inf \{G\mbox{-}dim_{Q} (M\otimes_{R} R') - G\mbox{-}dim_{Q} (R') :\\ &  R \rightarrow R' \leftarrow Q \ is\ a\ G\mbox{-}quasi\mbox{-}deformation\}.
		\end{align*} 
	
\end{defn}

Cohen-Macaulay local rings are characterized by the class of modules with finite Cohen-Macaulay dimensions. It is a formal repercussion of the more specific assertion below.
\begin{theorem}\cite[Theorem 3.9]{Gerko} \label{theorem_gerko} For any local ring $R$, the following statements are equivalent:
	\begin{enumerate}
			\item $R$ is Cohen-Macaulay.
			\item $CM\mbox{-}dim_{R} (M)  < \infty\ for\ all\ R\mbox{-}modules\ M.$
			\item $CM\mbox{-}dim_{R} (k) < \infty.$
		\end{enumerate}
	
\end{theorem}

\section{Proof of the main result}

Consider the following statements:
\vspace{2mm}

{\bf  P: } Let $R$ be a local ring of depth $1.$ If $(M,N)$ is any pair of non-zero Tor-independent $R$-modules with finite CM dimension, then $depth\, M + depth\, N = depth\, R + depth\, M \otimes_R N$ i.e. the pair $(M,N)$ satisfies the depth formula.

{\bf  Q: } Let $R$ be a local ring with depth $\geq 1.$ If $(M,N)$ is any pair of non-zero Tor-independent $R$-modules with finite CM dimension, then $depth\, M + depth\, N = depth\, R + depth\, M \otimes_R N$ i.e. the pair $(M,N)$ satisfies the depth formula.

To prove Theorem \ref{theorem_main} it suffices to show that $\textbf{P} \implies \textbf{Q}.$  For simplicity, we will divide {\bf P} into 3 simpler statements about local rings of depth $1$:
\begin{itemize}
	\item {\bf P0:} There exists no pair $(M,N)$ of Tor-independent $R$-modules with finite CM dimensions such that $depth\, M = depth\, N = 0.$ 
	\item {\bf P1:} If $(M,N)$ is a pair of Tor-independent $R$-modules with finite CM dimension $0,$ then $depth\, M \otimes_R N = 1.$  
	\item {\bf P2:} If $(M,N)$ is a pair of Tor-independent $R$-modules where $CM\mbox{-}dim(M) = 0,$ and $CM\mbox{-}dim(N) = 1,$ then $depth\, M \otimes_R N = 0.$ 
\end{itemize}

A simple argument (see Remark \ref{remark_1}) shows that {\bf P0}  $\implies$ {\bf P2}. Thus, we are only assuming {\bf P0} and {\bf P1}.

The proof of  {\bf P $\implies$ Q} is inductive and as an intermediate step, we will prove the following two special cases.
\begin{itemize}
	\item {\bf P1 $\implies $Q1} where {\bf Q1} is the statement: Let $R$ be a local ring with positive depth and $(M,N)$ be a pair of Tor-independent $R$-modules with $CM$ dimensions $0.$ Then $$depth\, R  = depth\, M \otimes_R N.$$ 
	\item  {\bf P1 + P2 $\implies $ Q2} where {\bf Q2} is the statement:  Let $R$ be a local ring with positive depth and $(M,N)$ be a pair of non-zero Tor-independent $R$-modules with $CM\mbox{-}dim(M) = 0,$ and $CM\mbox{-}dim(N) < \infty.$ Then $$depth\, N = depth\, M \otimes_R N.$$

	%
	%
	%
\end{itemize}

\begin{lemma} \label{lemma_depth}
	Let $M$ be an $R$-module with positive depth. Let $N$ be any ${R}$-module, such that $Tor_1^R(M,N) = 0.$ Then $depth\, M\otimes N > 0$ if and only if $ Tor_1^R(\overline{M}, N)  = 0$ for some $x$ which is an $M$-regular element.  
\end{lemma}
\begin{proof} If $depth\, M\otimes N > 0$, we can find $x \in m$ regular on $M$ and $M \otimes_R N.$ Tensoring the sequence $0 \rar M \xrightarrow{x} M \rar \overline{M} \rar 0$ with $N$ gives $$0 \rar  Tor_1^R(\overline{M}, N) \rar M \otimes N \xrightarrow{x} M \otimes N.$$ This proves one side.  Conversely, if $ Tor_1^R(\overline{M}, N)  = 0$ then the above sequence shows that $x$ is $M \otimes_R N$-regular .
\end{proof}

\begin{remark} \label{remark_1}
	We show that  {\bf P0}  $\implies$ {\bf P2}. Let $(M,N)$ be a pair of Tor-independent $R$-modules over a depth $1$ local ring $R$, such that $CM\mbox{-}dim(M) =0$ and $CM\mbox{-}dim(N) =1.$ If $depth\, M \otimes N > 0,$ then by Lemma \ref{lemma_depth}, there is some $x$  which is $M$-regular such that $Tor_1^R(\overline{M}, N) = 0.$ Thus $(\overline{M},N)$ are Tor-independent $R$-modules of depth $0$ and have finite CM dimension (by Theorem \ref{theorem_H_dim}), contradicting {\bf P0}.
\end{remark}

\begin{lemma} \label{lemma_spect}
	Let $(R,m,k)$ be a local ring with positive depth, and let $M$ and $N$ be non-zero $R$-modules.
	\begin{enumerate}
		\item If there exists an element  $x \in m$  which is regular on $R$ and $N$, then we have isomorphisms   $$Tor_p^{\overline{R}}(\overline{M}, \overline{N}) \cong Tor_p^R(\overline{M}, N),\,\,\,\, p \geq 1.$$
		\item If there exists an element  $x \in m$  which is regular on $R,$ then there exists a long exact sequence \begin{align*}  \cdots \rar Tor^R_{p+1}(N, \overline{M}) \rar  Tor_{p+1}^{\overline{R}}(\overline{N}, \overline{M}) \rar Tor_{p-1}^{\overline{R}}(Tor_1^R(N,\overline{R}), \overline{M}) \\ \rar Tor^R_{p}(N, \overline{M}) \rar  Tor_{p}^{\overline{R}}(\overline{N}, \overline{M}) \rar \cdots.
			\end{align*}
	\end{enumerate}
\end{lemma}
\begin{proof} There is a spectral sequence \cite[Theorem 5.6.6]{W} $$Tor_p^{\overline{R}}(Tor_q^{R}(N, \overline{R}), \overline{M}) \implies Tor^R_{p+q}(N, \overline{M}).$$
	
	\begin{enumerate}
		\item It follows by our assumptions that $Tor_q^R(N, \overline{R}) = 0$ for $q \geq 1.$ Therefore the spectral sequence collapses at $r=2$ to a single row and we get the claimed isomorphisms.
		\item It follows by our assumptions that $Tor_q^R(N, \overline{R}) = 0$ for $q \geq 2.$ Therefore the spectral sequence collapses at $r=2$ to the rows $q=0,1$ and we get the claimed long exact sequence.
	\end{enumerate} 
\end{proof}

\begin{corollary} \label{cor_spect} 
		Let $(R,m,k)$ be a local ring with positive depth and let $M$ and $N$ be non-zero Tor-independent $R$-modules. 
			\begin{enumerate}
				
			\item If there exists an element  $x \in m$  which is regular on $R,M$ and $N,$ then   \begin{align*}  Tor_p^{\overline{R}}(\overline{M}, \overline{N}) &= 0, \,\, p \geq 2, \\ Tor_1^{\overline{R}}(\overline{M}, \overline{N}) &\cong Tor_1^R(\overline{M}, N).
			\end{align*}
			\item If there exists an element  $x \in m$  which is regular on $R$ and $M,$ then there is an exact sequence \begin{align*} 0 \rar Tor_2^{\overline{R}}(\overline{N}, \overline{M}) \rar Tor_1^R(N,\overline{R}) \otimes \overline{M} \rar Tor_1^R(N,\overline{M}) \rar Tor_1^{\overline{R}}(\overline{N}, \overline{M}) \rar 0
			\end{align*} and isomorphisms $$Tor_{p}^{\overline{R}}(\overline{N}, \overline{M}) \cong Tor_{p-2}^{\overline{R}}(Tor_1^R(N,\overline{R}), \overline{M}),\ \  p \geq 3.$$
		\end{enumerate}
\end{corollary}
\begin{proof}
	Both claims follow from the vanishing of $Tor_p^R(\overline{M}, N)$ for $p \geq 2.$ 
\end{proof}

\begin{lemma}\label{lemma_tor_vanishing}
Let $M$ and $N$ be Tor-independent $R\mbox{-}$modules and let $x \in m$ be an element regular on $R, M, N$ and  $M\otimes_{R} N.$ Then $\overline{M}$ and $\overline{N}$ are Tor-independent over $\overline{R}.$ 
\end{lemma}

\begin{proof} By Lemma \ref{lemma_depth}, and tensoring $0 \rar M \xrightarrow{x} M \rar \overline{M} \rar 0$ with $N,$ we get that $\overline{M}$ and $N$ are Tor-independent over $R.$ Applying Corollary \ref{cor_spect} completes the proof.

\end{proof}

\begin{lemma} \label{lemma_depth_reduct}
	If there exists an element  $x \in m$  which is regular on $R$, $M$, $N$ and $M\otimes_{R} N.$ Then the following are equivalent
	\begin{enumerate}
		\item $M$ and $N$ are Tor-independent $R$-modules with $depth\, M = e$, $depth\, N = f$ and $depth\, M \otimes_R N = g.$
		\item $\overline{M}$ and $\overline{N}$ are Tor-independent $\overline{R}$-modules with $depth\, \overline{M} = e-1$, $depth\, \overline{N} = f-1$ , and $depth\, \overline{M} \otimes_{\overline{R}} \overline{N} = g-1.$
	\end{enumerate}
\end{lemma}
\begin{proof}
	$1 \implies 2$ is by Lemma \ref{lemma_tor_vanishing}. For the converse implication, the claim about depth is clear. By Lemma \ref{lemma_spect}, $Tor_p^R(\overline{M}, N) = 0$ for all $p \geq 1.$ Thus by Nakayama lemma, $M$ and $N$ are Tor-independent. 
\end{proof}

For notational convenience, we have used the subscript $d$ (for instance $M_d$ or $N_d$) to denote modules over a ring $R_d$ of depth $d,$ in the proof of Lemma \ref{lemma_depth_M_N_positive} and the proof of {\bf P1 $\implies $Q1}. 
\begin{lemma} \label{lemma_depth_M_N_positive}
	Let $R_d$ be a local ring with depth $d \geq 1.$ Let $(M_d,N_d)$ be a pair of Tor-independent $R_d$-modules with $CM\mbox{-}dim(M_d) = CM\mbox{-}dim(N_d) = 0.$ Then $${\bf P1} \implies depth\, M_d \otimes_{R_d} N_d > 0.$$ 
\end{lemma}
\begin{proof} By {\bf P1}, we may assume $d \geq 2.$ Let $\Omega_i(M_d)$ be the $i$'th syzygy of $M_d$ defined by a minimal free resolution $$  \cdots \rar P_{i,d} \rar \cdots \rar P_{1,d} \rar P_{0,d} \rar M_d \rar 0.$$ It is clear that $depth\, \Omega_i(M_d) \otimes_{R_d} N_d \geq i$ for $i \in \{1,2,\cdots, d\}.$ We work with the pair $(\Omega_{d-1}(M_d), N_d)$ of Tor-independent $R_d$-modules with (see Theorem \ref{theorem_H_dim}) $$CM\mbox{-}dim(\Omega_{d-1}(M_d)) = CM\mbox{-}dim(N_d) = 0$$ and $depth\, \Omega_{d-1}(M_d) \otimes_{R_d} N_d \geq d-1 \geq 1.$ Let $x_d \in m$  be a nonzero divisor on the finite collection of $R_d$ modules $\{R_d,M_d, N_d, \Omega_i(M_d), \{\Omega_i(M_d) \otimes_{R_d} N_d\}_{i \in \{1,2,\cdots, d-1\}}\}.$ For any $R$-module $T_d,$ define $T_{d-1} := T_d/x_dT_d.$ Then by Theorem \ref{theorem_H_dim}, the $R_{d-1}$-modules $M_{d-1}, N_{d-1},$ and $\{\Omega_i(M_{d-1})\}_{i \in \{1,2,\cdots, d-1\}}\}$ have Cohen-Macaulay dimension zero . By Lemma \ref{lemma_depth_reduct}, $$depth\, \Omega_{d-1}(M_{d-1})\otimes_{R_{d-1}} N_{d-1 }\geq d-2.$$ Continuing this way, by repeatedly applying Lemma \ref{lemma_depth_reduct} on successive pairs $(\Omega_{d-1}(M_i), N_i)$ for $i \in \{d-1,d-2,\cdots, 1\},$ we conclude that the pair $(\Omega_{d-1}(M_1), N_1)$ of $R_1$-modules is Tor-independent  with Cohen-Macaulay dimension zero, where $depth\, R_1 = 1.$
	
	By {\bf P1}, $depth\, \Omega_{d-1}(M_1) \otimes_{R_1} N_1= 1.$ Going upwards, by applying the converse implication of Lemma \ref{lemma_depth_reduct} on the pair $(\Omega_{d-1}(M_i), N_i),$ we conclude that $depth\, \Omega_{d-1}(M_d) \otimes_{R_d} N_d = d.$ The exact sequence $$0 \rar \Omega_{d-1}(M_d) \otimes_{R_d} N_d \rar P_{d-1,d} \otimes_{R_d} N_d \rar \cdots \rar P_{0,d} \otimes_{R_d} N_d \rar M_d \otimes_{R_d} N_d \rar 0,$$ shows that $depth\,  M_d \otimes_{R_d} N_d > 0.$

\end{proof}

\begin{proof}[Proof of {\bf P1 $\implies $Q1}] Without loss of generality, we may assume that $d = depth\, R_d \geq 2.$ We apply Lemma \ref{lemma_depth_M_N_positive} to conclude that $depth\, M_d \otimes_{R_d} N_d$ is positive. Thus by Lemma \ref{lemma_depth_reduct}, $M_{d-1}, N_{d-1}$ are Tor-independent $R_{d-1}$-modules with zero Cohen-Macaulay dimension. Applying Lemma \ref{lemma_depth_M_N_positive} on the pair of $R_{d-1}$-modules $(M_{d-1}, N_{d-1}),$ we get that $depth\, M_{d-1} \otimes_{R_{d-1}} N_{d-1}$ is positive. Continuing in this way, we conclude that $(M_1, N_1)$ is a pair of Tor-independent $R_1$-modules with zero Cohen-Macaulay dimension and with $depth\, M_1 \otimes_{R_1} N_1 = 1,$ where $depth\, R_1 = 1.$ Applying the converse implication of Lemma \ref{lemma_depth_reduct}, we work our way upwards, to conclude that $depth\, M_d \otimes_{R_d} N_d = d.$ 
\end{proof}

\begin{proof}[Proof of {\bf P1 + P2 $\implies $ Q2}] By assumption $CM\mbox{-}dim(M) = 0.$  We work by decreasing induction on the depth of $N.$ Since $CM\mbox{-}dim(N) < \infty$ therefore, by Theorem \ref{theorem_H_dim}, $depth\, N \leq depth\, R.$ If $depth\, N = depth\, R,$ then this holds by {\bf Q1} proved above, as then $M$ and $N$ will both be modules with zero Cohen-Macaulay dimension. 
Assume that $depth\, N < depth\, R.$ Let $\Omega_1(N)$ be the first syzygy of $N$ and $0 \rar \Omega_1(N) \rar Q_0 \rar N \rar 0$ be the associated sequence where $Q_0$ is a free $R$-module. Tensoring with $M$, we get $$0 \rar M \otimes \Omega_1(N) \rar M \otimes Q_0 \rar M \otimes N \rar 0.$$ Since $M$ and $\Omega_1(N)$ are Tor-independent and $depth\, \Omega_1(N) = depth\, N+1$, by inductive hypothesis, the pair $(M,\Omega_1(N))$ satisfies the depth formula. There are 2 possibilities
\begin{enumerate}[leftmargin=1cm]
	\item $depth\, M \otimes N < depth\, M \otimes Q_0 = depth\, R = d.$ Then we must have $depth\, M \otimes N = depth\, M \otimes \Omega_1(N) -1$ and we are done.  
	\item $depth\, M \otimes N \geq depth\, M \otimes Q_0 = d.$ Then, we must have $depth\, M \otimes \Omega_1(N) \geq d.$ 
	Therefore, by the inductive hypothesis, \begin{align*}
		d >&  depth\, N =  depth\, \Omega_1(N) -1 \\ =&  depth\, M \otimes \Omega_1(N) + depth\, R - depth\, M -1 \geq d-1.
	\end{align*} If $d \geq 2,$ by repeatedly applying Lemma \ref{lemma_depth_reduct}, we may assume $depth\, M = depth\, M \otimes N = depth\, R = 1$ and $depth\, N = 0.$ This contradicts the assumption {\bf P2}. 
\end{enumerate}
\end{proof}

\begin{lemma} \label{lemma_Gdim_inequality} Let $M$ and $N$ be $R$-modules such that $Tor_1^R(M,N) = 0.$ Then one of the following two holds 
	\begin{enumerate}
		\item $depth\, M \leq depth\,M \otimes N.$
		\item $depth\,M \otimes \Omega_1(N) = depth\,M \otimes N + 1.$
	\end{enumerate} 
\end{lemma}
\begin{proof} Assume that  $depth\, M > depth\,M \otimes N.$ Tensoring the long exact sequence \eqref{eqn_N} with $M,$ and the Tor vanishing gives  the short exact sequence $0 \rar M \otimes \Omega_1(N) \rar M \otimes Q_0 \rar M \otimes N \rar 0.$ By depth lemma, $depth\, M \otimes \Omega_1(N) = depth\, M \otimes N + 1$ which completes the proof. 
\end{proof}

\begin{theorem} \label{theorem_main_inequality} Let $R$ be a local ring of positive depth. Let $M$ and $N$ be Tor-independent $R$-modules such that $CM\mbox{-}dim_R \, M$ and $CM\mbox{-}dim_R \, N$ are finite. Assume that the statement ${\bf P}$ holds. Then $$depth\, R + depth\, M \otimes N \geq depth \, M + depth\, N.$$ 
\end{theorem}
\begin{proof} We work by decreasing induction on $depth\, N.$ If either $M$ or $N$ has zero Cohen-Macaulay dimension then the equality holds by {\bf Q2}, which is already proved. Thus, without loss of generality, we may assume that $\max\{depth\, M,depth\, N\} < depth\, R.$ Consider the following sequence \begin{align*} \label{seq_N_res}
	0 \rar\Omega_1(N) \rar Q_0 \rar N \rar 0.
	\end{align*} By depth lemma $depth\, \Omega_1(N) = depth\, N + 1.$ Further, $M$ and $\Omega_1(N)$ are non-zero Tor-independent $R$-modules. Inductively, the pair $(M,\Omega_1(N))$ satisfies the above inequality. If $depth\, M > depth\, M \otimes N,$ then by Lemma \ref{lemma_Gdim_inequality}, $depth\,M \otimes \Omega_1(N) = depth\,M \otimes N + 1$ and thus the pair $(M,N)$ also satisfies the inequality. 

	If $depth\, M \leq depth\, M \otimes N$, then using the fact that $depth\, N < depth\, R,$ we get a strict inequality $$depth\, M \otimes N + depth\, R > depth\, M + depth\, N.$$ Thus in either case, the inequality holds.

\end{proof}

\begin{corollary} \label{corollary_depth_reduct}
	Let $R$ be a local ring of positive depth. Let $M$ and $N$ be Tor-independent $R$-modules such that $depth\, R > \max\{depth\, M, depth\, N\}.$ Then there exists an element $x$ which is regular on $R$ and $\overline{R}$-modules $X$ and $Y$ which are Tor-independent, such that $depth\, M = depth\, X$ and $depth\, N = depth\, Y.$ 
\end{corollary}
\begin{proof}
We choose an element $x$ which is regular on $R, \Omega_1(M)$ and $\Omega_1(N).$ Define  $X =\overline{ \Omega_1(M)}$ and $Y = \overline{ \Omega_1(N)}.$ Since $depth\, \Omega_1(M) \otimes \Omega_1(N) > 0,$ we can apply Lemma \ref{lemma_depth_reduct} on the Tor-independent pair $(\Omega_1(M),\Omega_1(N)),$ to conclude that $(X,Y)$ are Tor-independent with $depth\, X = depth\, M$ and $depth\, Y = depth\, N.$ 
\end{proof}

\begin{lemma} \label{lemma_depth_0_general}
	Assume that ${\bf P0}$ holds. Let $R$ be a local ring of positive depth. Let $M,N$ be depth $0$ $R$-modules with finite Cohen-Macaulay dimensions. Then $M,N$ are not Tor-independent. 
\end{lemma}
\begin{proof} By repeated application of Corollary \ref{corollary_depth_reduct} and Theorem \ref{theorem_H_dim}, we reduce to the case $depth\, R = 1$ where it contradicts {\bf P0}.
\end{proof}

\begin{lemma} \label{lemma_main} Assume that ${\bf P}$ holds. Let $R$ be a local ring of positive depth. Let $M$ and $N$ be Tor-independent $R$-modules with finite Cohen-Macaulay dimensions. Then $$depth\, M + depth \, N - depth\, R \geq 0.$$ 
\end{lemma}
\begin{proof} We induct on $depth\, R.$ The claimed inequality holds by {\bf P} when $depth\, R = 1.$ Since {\bf P} $\implies${\bf Q2}, the inequality follows when $depth\, R = \max\{depth\, M,depth\, N\}.$ Thus, without loss of generality, we may assume that $depth\, R > depth\, M \geq depth\, N.$  By Lemma \ref{lemma_depth_0_general}, we may further assume that $depth\, M > 0.$ Thus $M$ and $\Omega_1(N)$ are Tor-independent $R$-modules with positive depth. The containment $M \otimes \Omega_1(N) \subset M \otimes Q_0 $ shows that $depth\, M\otimes \Omega_1(N)$ is positive. Let $x$ be an element which is regular on  $R, M, \Omega_1(N)$ and $M \otimes \Omega_1(N).$ By Lemma \ref{lemma_tor_vanishing}, $\overline{M}$ and $\overline{\Omega_1(N)} $ are Tor-independent $\overline{R}$-modules. Going modulo $x$, we get \begin{align*}
		depth\, \overline{M}+ depth\, \overline{\Omega_1(N)} - depth\, \overline{R} =& depth\, \overline{M}  + depth\, N - depth\, R + 1\\ = & depth\, M + depth\, N - depth\, R < 0.
	\end{align*} contradicting the inductive hypothesis.
\end{proof}

\begin{proof}[Proof of Theorem \ref{theorem_main}] Let $d \geq 2$ be the smallest integer for which there is a counterexample on a depth $d$ ring $R.$ By {\bf Q2}, we can assume that $depth\, R > \max\{depth\, M, depth\, N\}.$ Suppose that the inequality in Theorem \ref{theorem_main_inequality} is strict, then by Lemma \ref{lemma_main} we have $$depth\, M \otimes N > depth\, M + depth\, N - depth\, R \geq 0.$$ By Lemma \ref{lemma_depth_reduct}, we can reduce the depth and get a counterexample in lower depth which contradicts the minimality of $d.$
	
\end{proof}

\begin{proof}[Proof of Theorem \ref{corollary_main}] The dimension $0$ case is trivial.  For higher dimensions, we use Theorem \ref{theorem_gerko}  to show that the Cohen-Macaulay dimensions of the modules involved are finite and then apply Theorem \ref{theorem_main}. 
\end{proof}

\end{document}